\documentclass[11pt,a4paper]{amsart}
\usepackage[british]{babel}

\usepackage{a4wide}
\usepackage{amssymb}
\usepackage{amsmath}
\usepackage{amsthm}
\usepackage{array}
\usepackage{enumerate}
\usepackage[hidelinks]{hyperref}
\usepackage{breakurl}
\usepackage[backref=false,isbn=false,doi=true,backend=bibtex,style=alphabetic,maxbibnames=9]{biblatex}
\newtheorem{thm}{Theorem}[section]

\newtheorem{cor}[thm]{Corollary}
\newtheorem{lem}[thm]{Lemma}

\newtheorem{prop}[thm]{Proposition}

\theoremstyle{definition}

\renewcommand{\epsilon}{\varepsilon}
\renewcommand{\phi}{\varphi}
\newcommand{\defeq}{\mathrel{\mathop:}=}

\newcommand{\smsum}{\textstyle\sum}
\newcommand{\on}{\operatorname}

\begin{document}


\title{Every simple compact semiring is finite}

\author{Friedrich Martin Schneider}
\author{Jens Zumbr\"agel}
\address{F.~M.~Schneider, Institute of Algebra, TU Dresden, 01062 Dresden, Germany}
\address{J.~Zumbr\"agel, LACAL IC, EPFL, 1015 Lausanne, Switzerland}

\date{\today}

\begin{abstract} A Hausdorff topological semiring is called simple if every non-zero continuous homomorphism into another Hausdorff topological semiring is injective.  Classical work by Anzai and Kaplansky implies that any simple compact ring is finite. We generalize this result by proving that every simple compact semiring is finite, i.e., every infinite compact semiring admits a proper non-trivial quotient. \end{abstract}

\maketitle


\section{Introduction}

In this note we study simple Hausdorff topological semirings, i.e., those where every non-zero continuous homomorphism into another Hausdorff topological semiring is injective. A compact Hausdorff topological semiring is simple if and only if its only closed congruences are the trivial ones. The structure of simple compact rings is well understood: a classical result due to Kaplansky~\cite{kaplansky-rings} states that every simple compact Hausdorff topological ring is finite and thus -- by the Wedderburn-Artin theorem -- isomorphic to a matrix ring $\on{M}_{n}(\mathbb{F})$ over some finite field~$\mathbb{F}$. In particular, it follows that any compact field is finite. We note that Kaplansky's result may as well be deduced from earlier work of Anzai~\cite{anzai}, who proved that every compact Hausdorff topological ring with non-trivial multiplication is disconnected and that moreover every compact Hausdorff topological ring without left (or right) total zero divisors is profinite, i.e., representable as a projective limit of finite discrete rings. Of course, a generalization of the mentioned results by Anzai cannot be expected for general compact semirings: in fact, there are numerous examples of non-zero connected compact semirings with multiplicative unit, which in particular cannot be profinite. However, in the present paper, we extend Kaplansky's result and show that any simple compact Hausdorff topological semiring is finite (Theorem~\ref{theorem:main.result}). Hence, the classification of finite simple semirings applies, which has been established in~\cite{zumba}.

\section{Semirings}\label{section:preliminaries}

In this section we briefly recall several elementary concepts concerning semirings. For a start let us fix some general terminology. We assume the reader to be familiar with classical algebraic structures or algebras, such as semigroups, monoids, groups, and rings, as well as the related concepts of subalgebras, homomorphisms, and product algebras. If~$A$ is any algebra, then a \emph{congruence} on~$A$ is an equivalence relation on~$A$ constituting a subalgebra of the product algebra~$A \times A$. 

Let~$R$ be a \emph{semiring}, i.e., an algebra $(R,+,\cdot,0)$ satisfying the following conditions: \begin{itemize}
	\item[---] \ $(R,+,0)$ is a commutative monoid and $(R,\cdot)$ is a semigroup,
	\item[---] \ $x \cdot (y + z) = x\cdot y + x\cdot z$ and $(x + y)\cdot z = x \cdot z + y \cdot z$ for all $x,y,z \in R$,
	\item[---] \ $0 \cdot x = x \cdot 0 = 0$ for every $x \in R$.
\end{itemize} A \emph{subsemiring} of~$R$ is a subset $A \subseteq R$ such that~$A$ is a submonoid of the additive monoid of~$R$ and a subsemigroup of the multiplicative semigroup of~$R$. Similarly, a \emph{homomorphism} from~$R$ into another semiring~$S$ is a map $h \colon R \to S$ such that~$h$ is both a homomorphism from the additive monoid of~$R$ to that of~$S$ and from the multiplicative semigroup of~$R$ to that of~$S$. Clearly, an equivalence relation~$\theta$ on~$R$ is a congruence on~$R$ if and only if~$\theta$ is a congruence on the additive monoid of~$R$ and the multiplicative semigroup of~$R$. As usual, an \emph{ideal} of~$R$ is a submonoid~$A$ of the additive monoid of~$R$ such that $R A \cup A R \subseteq A$. Moreover, a subset $A \subseteq R$ is called \emph{subtractive} if the following holds: \begin{displaymath}
	\forall x \in R \, \forall a \in A \colon \ x+a \in A \Longrightarrow x \in A .
\end{displaymath}
Subtractive ideals are closely related to the following congruences due to Bourne~\cite{bourne51}, as the subsequent basic lemma reveals.

\begin{lem}[\cite{bourne51}]\label{lemma:bourne} Let $R$ be a semiring and let $A$ be an ideal of $R$. Then \begin{displaymath}
	\kappa_{A} \defeq \{ (x,y) \in R \times R \mid \exists a,b \in A \colon \, x+a = y+b \}
\end{displaymath} is a congruence on $R$. Furthermore, $A$ is subtractive if and only if $A = [0]_{\kappa_{A}}$. \end{lem}

Let us turn our attention towards naturally ordered semirings. To this end, let $(M,+,0)$ be a commutative monoid. Notice that the relation given by
 \begin{equation*}
	x \leq y \ :\Longleftrightarrow \ \exists z \in M \colon\, x + z = y \qquad (x,y \in M)
\end{equation*} is a preorder, i.e., $\leq$ is reflexive and transitive. We say that~$M$ is \emph{naturally ordered} if the preorder~$\le$ is anti-symmetric, which means that $(M,\leq)$ is a partially ordered set with least element~$0$. Now let~$R$ be a semiring and consider the preorder~$\leq$ defined as above with regard to the additive monoid of~$R$. It is straightforward to check that \begin{displaymath}
	\forall x,x',y,y' \in R \colon \ x \leq x', \, y \leq y' \, \Longrightarrow \, x+y \leq x'+y' , \ xy \leq x'y' .
\end{displaymath} We say that~$R$ is \emph{naturally ordered} if its additive monoid is naturally ordered. Such a semiring~$R$ is called \emph{bounded} if the partially ordered set $(R,\leq)$ is bounded, in which case we denote the greatest element by~$\infty$. One may easily deduce the following observations. 

\begin{lem}\label{lemma:downsets} Let~$R$ be a semiring and let $a \in R$. Then the following hold: \begin{enumerate}[\quad\upshape (1)]
\item $I_a \defeq \{ x \in R \mid x \leq a \}$ is a subtractive set.
\item $I_a$ is a submonoid of $(R,+,0)$ if and only if $a + a \le a$.
\item If~$R$ is naturally ordered and bounded, then $R I_a \cup I_a R \subseteq I_a$ if and only if $\infty a \infty \le a$.
\end{enumerate} \end{lem}

\section{Topological semirings}\label{section:topological.preliminaries}

Henceforth, our primary concern are \emph{topological algebras}, i.e., algebras coming along with a topology such that each of the operations of the algebra is continuous. We are particularly interested in their \emph{closed} congruences, that is, those congruences on a topological algebra $A$ which are closed in the product space $A \times A$. In this regard, we study \emph{topological semirings}, i.e., semirings equipped with a topology so that both the addition and the multiplication are continuous maps. The study of topological semirings was initiated by Bourne~\cite{bourne59,bourne60}, and a list of further references may be found in~\cite{robbie-survey}. 

As usual, by a \emph{topological ring} we mean a ring equipped with a topology such that the additive group is a topological group and the multiplicative semiring is a topological semiring. Recall that any compact Hausdorff topological monoid that is a group is indeed a topological group (see, e.g.,~\cite{topgroups}). This readily implies the following fact.

\begin{prop}\label{proposition:compact.rings} Let~$R$ be a compact Hausdorff topological semiring.  If~$R$ is a	ring, then~$R$ is a topological ring. \end{prop}

Let us make a simple remark concerning the Bourne congruence mentioned in Section~\ref{section:preliminaries}.

\begin{lem}\label{lemma:closed.bourne} Let $R$ be a compact Hausdorff topological semiring. If $A$ a closed ideal of $R$, then the congruence $\kappa_{A}$ is closed in $R \times R$. \end{lem}

\begin{proof} Obviously, each of the two mappings $\lambda \colon R^2 \times A^2 \to R^2, \, (x, y, a, b) \mapsto (x+a, y+b)$ and $\pi \colon R^2 \times A^2 \to R^2, \, (x, y, a, b) \mapsto (x,y)$ is continuous. Since~$R$ is a Hausdorff space, the set $\Delta_R = \{ (x,x) \mid x \in R \}$ is closed in $R \times R$. As~$R$ and hence~$A$ are compact, it follows that the closed subset $\lambda^{-1} (\Delta_{R}) \subseteq R^2 \times A^2$ is compact. Therefore $\kappa_{A} = \pi( \lambda^{-1} (\Delta_{R}) )$ is compact and thus closed in $R \times R$. \end{proof}

Next we observe that for compact semirings the preorder introduced in Section~\ref{section:preliminaries} interacts nicely with the given topology. 

\begin{lem}\label{lemma:closed.downsets} If~$R$ is a compact Hausdorff topological semiring, then $\leq$ is closed in $R \times R$. \end{lem}

\begin{proof} The proof proceeds analogously to that of Lemma~\ref{lemma:closed.bourne}. Considering the continuous mappings given by $\phi \colon R^3 \to R^2, \, (x, y, z) \mapsto (x+z, y)$ and $\pi \colon R^3 \to R^2, \, (x, y, z) \mapsto (x, y)$, we conclude that ${\le} = \pi( \phi^{-1}( \Delta_R) )$ is closed in $R \times R$. \end{proof}

As one might expect, compactness affects the order structure of a naturally ordered Hausdorff topological semiring (see Corollary~\ref{corollary:compact.semirings.are.complete}). In fact, this is due to a more general reason as the following observation reveals.

\begin{prop}\label{proposition:bounded} If $(M,+,0)$ is a naturally ordered commutative compact Hausdorff topological monoid, then $(M,\leq)$ is bounded. \end{prop}

\begin{proof} For a finite subset $F \subseteq M$, denote by~$A_{F}$ the closure of $\{ \smsum F' \mid F' \subseteq M \text{ finite}, \, F \subseteq F' \}$ in the topological space~$M$. Now $\mathcal{A} := \{ A_{F} \mid F \subseteq M \text{ finite} \}$ is a collection of closed non-empty subsets of $M$. Since $A_{F_{0} \cup F_{1}} \subseteq A_{F_{0}} \cap A_{F_{1}}$ for all finite subsets $F_{0},F_{1} \subseteq S$, we conclude that $\mathcal{A}$ has the finite intersection property. Thus, $\bigcap \mathcal{A} \ne \varnothing$ by compactness of~$M$. Consider any $s \in \bigcap \mathcal{A}$. Then $s \in A_{\{ x \}}$ and hence $x \leq s$ for every $x \in M$. This shows that~$s$ is the greatest element of $(M,\leq)$. \end{proof}

\begin{cor}\label{corollary:compact.semirings.are.complete} Any naturally ordered compact Hausdorff topological semiring is bounded. \end{cor}

\section{Simple compact semirings}\label{section:classification}

In this section we finally come to simple compact semirings. A Hausdorff topological algebra~$A$ is called \emph{simple} if every non-constant continuous homomorphism from~$A$ into another Hausdorff topological algebra of the same type is injective. We start with a simple reformulation of this property.

\begin{prop}\label{proposition:simple.equivalence} Let~$A$ be a compact Hausdorff topological algebra. Then~$A$ is simple if and only if $\Delta_A = \{ (a, a) \mid a \in A \}$ and $A \times A$ are the only closed congruences on~$A$. \end{prop}

\begin{proof} ($\Longrightarrow$) Suppose that $A$ is simple. Let $\theta$ be a closed congruence on~$A$. Then there is a unique algebraic structure on the quotient set $A/\theta = \{ [a]_{\theta} \mid a \in A \}$ such that the quotient map $\phi \colon A \to A/\theta, \, a \mapsto [a]_{\theta}$ becomes a homomorphism. We endow $A/\theta$ with the corresponding quotient topology, i.e., the final topology generated by $\phi$. Since $A$ is a compact Hausdorff space and $\theta$ is closed in $A \times A$, the quotient space $A/\theta$ is a Hausdorff space (see~\cite[\S 10.4, Prop.~8]{Bourbaki1}). It follows that $A^{n} \to (A/\theta)^{n}, \, (a_{1},\ldots,a_{n}) \mapsto ([a_{1}]_{\theta},\ldots,[a_{n}]_{\theta})$ is a quotient map for every $n \geq 1$ (see~\cite[\S 10.2, Cor.~2]{Bourbaki1}). From this and the fact that $\phi$ is a continuous homomorphism, we infer that any of the operations of the algebra $A/\theta$ is continuous. This shows that $A/\theta$ is a Hausdorff topological algebra. Hence, simplicity of~$A$ asserts that $\phi$ is injective or constant. Consequently, $\theta = \Delta_{A}$ or $\theta = A \times A$.

	($\Longleftarrow$) Assume that $\Delta_A$ and $A \times A$ are the only closed congruences on~$A$. Let $B$ be any Hausdorff topological algebra of the same type as $A$ and let $\phi \colon A \to B$ be a non-constant continuous homomorphism. Since $B$ is a Hausdorff space, $\Delta_{B}$ is a closed congruence on $B$. As $\phi$ is a continuous homomorphism, $\ker \phi = \{ (x,y) \in A \times A\mid \phi(x) = \phi(y) \} = (\phi \times \phi)^{-1}(\Delta_{B})$ is a closed congruence on $A$. By assumption, $\ker \phi = \Delta_{A}$ or $\ker \phi = A \times A$. Since~$\phi$ is non-constant, it follows that~$\phi$ is injective. \end{proof}

Note that a Hausdorff topological ring~$R$ is simple if and only if~$\{ 0 \}$ and $R$ are the only closed ideals of~$R$. We recall the following classical result due to Kaplansky~\cite{kaplansky-rings}.

\begin{thm}[\cite{kaplansky-rings}]\label{theorem:simple.compact.rings} Every simple compact Hausdorff topological ring is finite. \end{thm}

We remark that one may deduce Kaplansky's result from the fact that any compact Hausdorff topological ring with non-trivial multiplication is disconnected, which was proven earlier by Anzai~\cite{anzai}. Even though there are many examples of non-trivial connected compact Hausdorff topological semirings, it is natural to ask whether Theorem~\ref{theorem:simple.compact.rings} extends to semirings. We provide an affirmative answer to this question in Theorem~\ref{theorem:main.result}, which constitutes the main result of the present note.

As a first step, we give the following necessary condition for a compact Hausdorff topological semiring to be simple, which follows immediately from Lemma~\ref{lemma:bourne}, Lemma~\ref{lemma:closed.bourne}, and Proposition~\ref{proposition:simple.equivalence}.

\begin{cor}\label{corollary:ideals.in.simple.rings} If~$R$ is a simple compact Hausdorff topological semiring, then $\{ 0 \}$ and $R$ are the only closed subtractive ideals of $R$. \end{cor}

The next result provides a key observation towards the desired Theorem~\ref{theorem:main.result}.

\begin{prop}\label{proposition:dichotomy} If~$R$ is a simple compact Hausdorff topological semiring, then~$R$ is a finite ring or a naturally ordered semiring. \end{prop}

\begin{proof} It is easy to check that the equivalence relation $\theta \defeq \{ (x,y) \in R \times R \mid x \leq y , \, y \leq x \}$ is a congruence on~$R$. Furthermore, Lemma~\ref{lemma:closed.downsets} implies that $\theta$ is closed in $R \times R$. As~$R$ is simple, Proposition~\ref{proposition:simple.equivalence} asserts that $\theta = {\Delta_{R}}$ or $\theta = {R \times R}$. In the first case, $R$ is naturally ordered. In the second case, $(x,0) \in \theta$ and hence $0 \in R + x$ for all $x \in R$, which means that the additive monoid of $R$ is a group.  Consequently, $R$ is a ring and therefore $R$ is a topological ring by Proposition~\ref{proposition:compact.rings}, which is finite according to Theorem~\ref{theorem:simple.compact.rings}. \end{proof}

The previous result reduces the task of showing that all simple compact semirings are finite to studying naturally ordered ones.  As an additional preparatory step we state the following result, which could also be proven in a more general context, cf.~\cite[Thm.~4.4, Proof]{profinite}.

\begin{lem}\label{lemma:open-congruence} Let~$R$ be a compact Hausdorff topological semiring. If~$\theta$ is an open equivalence relation on~$R$, then there exists an open congruence~$\theta_0$ on~$R$ such that $\theta_0 \subseteq \theta$. \end{lem}

\begin{proof} Consider the continuous map $\Phi \colon R^{5} \to R^{8}$ defined by \begin{displaymath} \Phi(r,s,t,x,y) \defeq (rxs+t, \, rys+t, \, rx+t, \, ry+t, \, xs+t, \, ys+t, \, x+t, \, y+t) .
\end{displaymath} It is easy to check that the equivalence relation \begin{displaymath}
	\theta_{0} \defeq \{ (x, y) \in R \times R \mid \forall r, s, t \in R \colon \, \Phi(r,s,t,x,y) \in \theta^4 \}
\end{displaymath} is a congruence on~$R$ contained in~$\theta$, where we regard $\theta^4$ as a subset of~$R^8$ in the obvious way. We argue that $\theta_{0}$ is open in $R \times R$. To see this, let $(x,y) \in \theta_{0}$. Then $\Phi (R^{3} \times \{ (x,y) \}) \subseteq \theta^{4}$. As~$R^{3}$ is compact and $\theta^4$ is open in~$R^8$, there is an open subset $U \subseteq R \times R$ such that $(x,y) \in U$ and $\Phi (R^{3} \times U) \subseteq \theta^{4}$. Thus, $(x,y) \in U \subseteq \theta_{0}$. This proves that~$\theta_{0}$ is open in $R \times R$. \end{proof}

Now everything is prepared to prove the aforementioned main result.

\begin{thm}\label{theorem:main.result} Every simple compact Hausdorff topological semiring is finite. \end{thm}

\begin{proof} Let $R$ be a simple compact Hausdorff topological semiring. By Proposition~\ref{proposition:dichotomy}, we may assume that~$R$ is naturally ordered and thus bounded due to Corollary~\ref{corollary:compact.semirings.are.complete}.

We first treat the case where $R R = \{ 0 \}$. Let $a \in R$ and consider the equivalence relation \begin{displaymath}
	\rho_a \defeq \{ (x, y) \in R \times R \mid x = y \,\text{ or }\, x, y \ge a \} ,
\end{displaymath} which is easily seen to be a semiring congruence as $R R = \{ 0 \}$. Furthermore, $\rho_a$ is closed by Lemma~\ref{lemma:closed.downsets}. Hence, $\rho_a = \Delta_R$ or $\rho_a = R \times R$ due to simplicity of~$R$, which means that $a = \infty$ or $a = 0$. Therefore, $R = \{ 0, \infty \}$ is finite.

Henceforth suppose that $R R \ne \{ 0 \}$. We argue that $\infty \infty = \infty$. To this end, let $a := \infty \infty$. Then $a + a \le a$ and $\infty a \infty \le a$, thus $I_a$ is a subtractive closed ideal by Lemma~\ref{lemma:downsets} and Lemma~\ref{lemma:closed.downsets}. As $R R \ne \{ 0 \}$, we have $I_a \neq \{ 0 \}$, and so the claim follows from Corollary~\ref{corollary:ideals.in.simple.rings}. 
Now let us consider the continuous mapping \begin{displaymath}
	\phi \colon R \to R \, , \quad x \mapsto \infty x \infty .
\end{displaymath} Evidently, $\phi(0) = 0$ and we have seen that $\phi(\infty) = \infty$. Consider any $x \in R$ and notice that $\phi(x) + \phi(x) \leq \phi(x)$ and $\infty \phi(x) \infty \leq \phi(x)$. Hence, $I_{\phi(x)}$ is a subtractive ideal of~$R$ due to Lemma~\ref{lemma:downsets}. Besides, Lemma~\ref{lemma:closed.downsets} asserts that $I_{\phi (x)}$ is closed in~$R$. Therefore $I_{\phi (x)} = \{ 0 \}$ or $I_{\phi (x)} = R$ by Corollary~\ref{corollary:ideals.in.simple.rings}. In particular, the image $\phi (R) = \{ 0, \infty \}$ is finite and thus discrete. Consequently, $\theta \defeq \ker \phi$ is an open equivalence relation on~$R$. By Lemma~\ref{lemma:open-congruence}, there exists an open congruence $\theta_{0} \subseteq \theta$. Since $R/\theta_{0}$ is a partition of~$R$ into open subsets and~$R$ is compact, it follows that $R/\theta_{0}$ is finite. Moreover, as any open equivalence relation is also closed, Proposition~\ref{proposition:simple.equivalence} necessitates that $\theta_{0} = R \times R$ or $\theta_{0} = \Delta_R$. The first case implies that $0 = \phi(0) = \phi(\infty) = \infty$ and therefore $R = \{ 0 \}$, while in the second case $R \cong R/\theta_{0}$, which shows that $R$ is finite. \end{proof}

\section*{Acknowledgments}

The first author is supported by funding of the Excellence Initiative by the German Federal and State Governments. The second author has been funded by the Irish Research Council under grant no.~ELEVATEPD/2013/82.

\printbibliography

\end{document}